\documentclass[12pt]{amsart}
 
\pdfoutput=1
\usepackage[margin=1in]{geometry}
\usepackage{amsmath,amsthm,amssymb,bbm,color,esint,esvect,float,graphicx,mathrsfs}
\usepackage{soul}
\usepackage[abbrev]{amsrefs}

\usepackage{euscript,latexsym}
\usepackage{accents}

\usepackage{epstopdf}
\AppendGraphicsExtensions{.tif}
\newcommand{\N}{\mathbb{N}}
\newcommand{\Z}{\mathbb{Z}}
\newcommand{\R}{\mathbb{R}}

 \newcommand{\vertiii}[1]{{\left\vert\kern-0.25ex\left\vert\kern-0.25ex\left\vert #1 
    \right\vert\kern-0.25ex\right\vert\kern-0.25ex\right\vert}}

 \newtheorem{thm}{Theorem}[section]
 \newtheorem{lemma}[thm]{Lemma}
 \newtheorem{cor}[thm]{Corollary}
 \newtheorem{prop}[thm]{Proposition}
 \newtheorem{rem}[thm]{Remark}
 \newtheorem{defin}[thm]{Definition}

 \numberwithin{equation}{section}
 
\newenvironment{theorem}[2][Theorem]{\begin{trivlist}
\item[\hskip \labelsep {\bfseries #1}\hskip \labelsep {\bfseries #2.}]}{\end{trivlist}}

\theoremstyle{definition}

\begin{document}

\title[Compact pseudodifferential and Fourier integral operators via localization]{Compact pseudodifferential and Fourier integral operators via localization}

\author{Cody B. Stockdale}
\address{Cody B. Stockdale, School of Mathematical Sciences and Statistics, Clemson University, Clemson, SC 29634, USA}
\email{cbstock@clemson.edu}

\author{Cody Waters}
\address{Cody Waters, School of Mathematical Sciences and Statistics, Clemson University, Clemson, SC 29634, USA}
\email{chwater@clemson.edu}

\begin{abstract}
We present a general framework of localized operators, i.e., operators whose matrix coefficients with respect to the Gabor frame are concentrated on the diagonal. We show that localized operators are bounded between modulation spaces, and we deduce their compactness from an easily verifiable weak compactness condition. We apply this abstract formalism to unify and extend existing theorems for pseudodifferential and Fourier integral operators, and to obtain new results for three-parameter pseudodifferential operators.
\end{abstract}



\maketitle


\section{Introduction}\label{Section:Introduction}

Pseudodifferential operators are ubiquitous in analysis and have widespread applications in fields such as partial differential equations, signal processing, and quantum mechanics. In their seminal work \cite{CV1972}, Calder\'on and Vaillancourt famously showed that if a symbol has bounded derivatives up to a certain finite order, 
then the associated pseudodifferential operator is bounded on $L^2(\mathbb{R}^n)$. This result had a profound impact and inspired many refinements and extensions; see, for instance, \cites{S1993, G2001, CR2020} and references therein. 

The required smoothness of the symbol for the associated pseudodifferential operator to be bounded was more precisely quantified in \cite{HRT1997}, where it was shown that if $s>2n$ and $\sigma$ is in the H\"older-Zygmund space $\Lambda^s(\mathbb{R}^{2n})$, then the Kohn-Nirenberg and Weyl pseudodifferential operators, $Op^{\text{KN}}(\sigma)$ and $Op^{\text{W}}(\sigma)$, are bounded on $L^2(\mathbb{R}^n)$. Independently, in \cite{S1994}, Sj\"ostrand further relaxed the smoothness assumption on $\sigma$ to membership in $M^{\infty,1}(\mathbb{R}^{2n})$. Under this assumption, 
the bounds extend from $L^2(\mathbb{R}^n)$ to general modulation spaces $M^{p,q}(\mathbb{R}^{n})$ and hold for any $\tau$-pseudodifferential operator $Op_{\tau}(\sigma)$ given by 
$$
    Op_{\tau}(\sigma)f(x) = \iint_{\mathbb{R}^{2n}} \sigma(\tau x+(1-\tau)y,\xi)e^{2\pi i (x-y)\cdot\xi}f(y)\,dyd\xi.
$$
Recall the Kohn-Nirenberg and Weyl pseudodifferential operators are the cases of $Op_{\tau}(\sigma)$ when $\tau=1$ and $\tau=\frac{1}{2}$, respectively. 
\begin{theorem}{A}
\emph{Let $p,q\in[1,\infty)$ and $\tau \in [0,1]$. If $\sigma \in M^{\infty,1}(\mathbb{R}^{2n})$, then $Op_{\tau}(\sigma)$ is bounded on $M^{p,q}(\mathbb{R}^n)$.}
\end{theorem}
\noindent In light of the $\tau$-independence of \cite{S1994}*{Corollary 1.2}, the bound on $M^{2,2}(\mathbb{R}^n)=L^2(\mathbb{R}^n)$ was proved in \cite{S1994}*{Section 3}, the $p=q$ case of Theorem A first appeared in \cite{GH1999}*{Theorem 1.1}, and the full result was established by Gr\"ochenig in \cite{G2001}*{Theorem 14.5.2}. 

Compact pseudodifferential operators have previously been studied in \cites{BGHO2005, FG2007, CST2023, C1975, DR2016, L2001, M2011, M2014, MW2010}. In particular, the compactness of $Op_{\tau}(\sigma)$ with $\sigma \in M^{\infty,1}(\mathbb{R}^{2n})$ on $M^{p,q}(\mathbb{R}^{n})$ is related to the vanishing of its symbol's short-time Fourier transform (STFT) $V_{\varphi}\sigma$ as follows. 
\begin{theorem}{B}
\emph{Let $p,q\in [1,\infty)$ and $\tau \in [0,1]$. 
If $\sigma \in M^{\infty,1}(\mathbb{R}^{2n})$ satisfies 
$$
    \lim_{(z,\zeta)\rightarrow \infty} V_{\varphi}\sigma(z,\zeta) = 0,
$$
then $Op_{\tau}(\sigma)$ is compact on $M^{p,q}(\mathbb{R}^{n})$.}
\end{theorem}
\noindent The sufficiency of the vanishing of the STFT for compactness of $Op^{\text{KN}}(\sigma)$ was given by B\'enyi, Gr\"ochenig, Heil, and Okoudjou in \cite{BGHO2005}*{Proposition 2.3}, and its necessity for the compactness of $Op^{\text{W}}(\sigma)$ was proved by Fern\'andez and Galbis in \cite{FG2007}*{Theorem 4.6}. 

We recognize Theorems A and B as particular instances of a general phenomenon regarding localized operators. The notion of localization, which quantifies the idea that the matrix coefficients of an operator with respect to the Gabor frame are concentrated on the diagonal (actually, on the graph of a bi-Lipschitz diffeomorphism), was introduced in the context of Weyl pseudodifferential operators by Gr\"ochenig in \cite{G2006}; in fact, localization of $Op^{W}(\sigma)$ is a defining feature of $\sigma \in M^{\infty,1}(\mathbb{R}^{2n})$. We show that localized operators are bounded on modulation spaces, and we characterize their compactness with a weak compactness condition, which asserts that the matrix coefficients decay in the direction of the diagonal. 

Our main result makes this idea precise in a more general setting and is stated as follows. 
\begin{thm}\label{AbstractCompactnessTheorem}
Let $p,q \in [1,\infty)$, $g_1,g_2 \in \mathcal{S}_{\vartheta}(\mathbb{R}^n) \setminus \{ 0 \}$, $\chi:\mathbb{R}^{2n}\rightarrow\mathbb{R}^{2n}$ be a bi-Lipschitz diffeomorphism, $\nu$ be an admissible weight, and $m$ be a $\nu$-moderate weight. If $T : \mathcal{S}_{\vartheta}(\mathbb{R}^n) \rightarrow \mathcal{S}_{\vartheta}'(\mathbb{R}^n)$ is $(\nu,\chi,g_1,g_2)$-localized, then $T$ is bounded from $M^{p,q}_{m,\chi}(\mathbb{R}^n)$ to $M^{p,q}_{m}(\mathbb{R}^n)$. If, additionally, $T$ is $(\chi,g_1,g_2)$-weakly compact, then $T$ is compact from $M^{p,q}_{m,\chi}(\mathbb{R}^n)$ to $M_m^{p,q}(\R^n)$. 
\end{thm}
\noindent See Section \ref{AbstractCompactnessSection} for details on the setup and definitions involved in Theorem \ref{AbstractCompactnessTheorem}. 

\begin{rem}
    While the extension of Calder\'on and Vaillancourt's result of Theorem A 
    applies to all $Op_{\tau}(\sigma)$ (and, as we will see below, even more general operators) and yields bounds on all modulation spaces, it only gives $L^p(\mathbb{R}^n)$ bounds for $p=2$. Other versions concern symbols in H\"ormander's class $S^m_{\rho,\delta}(\mathbb{R}^{2n})$ and yield $L^p(\mathbb{R}^n)$ bounds for a full range of $p$. In particular, if $\sigma \in S_{1,0}^0(\mathbb{R}^{2n})$, i.e. $\sigma$ is smooth and satisfies 
$$
    | \partial_x^{\alpha} \partial_\xi^{\beta} \sigma(x,\xi)| \lesssim (1+|\xi|)^{-|\beta|}
$$
    for all multi-indices $\alpha$ and $\beta$, 
    then $Op^{\text{KN}}(\sigma)$ is a Calder\'on-Zygmund operator and hence bounded on $L^p(\mathbb{R}^n)$ for all $p \in (1,\infty)$; see \cite{S1993}. Note that the condition $\sigma \in M^{\infty,1}(\mathbb{R}^{2n})$ is much weaker; namely, if $s>2n$, then 
    $$
        S^0_{1,0}(\mathbb{R}^{2n}) \subsetneq \Lambda^s(\mathbb{R}^{2n}) \subsetneq M^{\infty,1}(\mathbb{R}^{2n}).
    $$    
    
    The compactness of such operators was recently investigated by Carro, Soria, and Torres in \cite{CST2023}, where it was shown that if $\sigma$ satisfies a vanishing version of the $S_{1,0}^0(\mathbb{R}^{2n})$ condition, then $Op^{\text{KN}}(\sigma)$ is a compact Calder\'on-Zygmund operator and hence compact on $L^p(\mathbb{R}^n)$ for all $p \in (1,\infty)$. 
    See \cites{V2015, SVW2022, FGW2023, BLOT2025, BLOT20251, OV2017, PPV2017, CYY2024, MS2023, BOT2024} for more on compact Calder\'on-Zygmund theory. 
\end{rem}

\begin{rem}
    In personal communication \cite{F2024}, Fulsche described how to achieve a result similar to the $p=q$ case of Theorem B using quantum harmonic analysis techniques. Let 
    $$
        C(L^2(\mathbb{R}^n)):= \{T \in \mathcal{B}(L^2(\mathbb{R}^n)): \pi(z)T\pi(z)^* \rightarrow T \text{  in operator norm as  } z\rightarrow 0\},
    $$
    where $\pi(z)$ denotes the time-frequency shift defined in Section \ref{AbstractCompactnessSection} below. A foundational result of Werner \cite{W1984}*{Corollary 5.1} states that if $T \in \mathcal{B}(L^2(\mathbb{R}^n))$, then $T$ is compact if and only if $T \in C(L^2(\mathbb{R}^n))$ and $R_0 * T$ is continuous and vanishes at infinity, where $R_0 := \phi \otimes \phi$ with $\phi$ being the ground state of the quantum harmonic oscillator and $R_0 * T$ denotes the operator convolution. One can show that $Op_{\tau}(\sigma) \in C(L^2(\mathbb{R}^n))$ for $\sigma \in M^{\infty,1}(\mathbb{R}^{2n})$, and so the following holds as a consequence: if $\tau \in [0,1]$ and $\sigma \in M^{\infty,1}(\mathbb{R}^{2n})$, then $Op_{\tau}(\sigma)$ is compact on $L^2(\mathbb{R}^n)$ if and only if $\sigma * \varphi$ vanishes at infinity, where $\varphi$ is the standard Gaussian. With suitable adaptations, one can replace $L^2(\mathbb{R}^n)$ with $M^{p,p}(\mathbb{R}^n)$ for $p \in [1,\infty)$; indeed, this approach was very recently taken in \cite{FH2025}*{Theorem 3.11}. 
\end{rem}

We apply Theorem \ref{AbstractCompactnessTheorem} to unify and extend known theorems for pseudodifferential and Fourier integral operators, and to obtain new results for three-parameter pseudodifferential operators. 

\subsection{\boldmath${\tau}$\unboldmath-pseudodifferential operators}
Our first application of Theorem \ref{AbstractCompactnessTheorem} regards $\tau$-pseudodifferential operators with symbols in weighted Sj\"ostrand classes.

\begin{thm}\label{TauPDOCompactness1}
    Let $p,q \in [1,\infty)$, $\nu$ be an admissible weight, $m$ be a $\nu$-moderate weight, and $\tau \in [0,1]$. If $\sigma \in M^{\infty,1}_{1\otimes \nu\circ J^{-1}}(\R^{2n})$, then $Op_{\tau}(\sigma)$ is bounded on $M^{p,q}_m(\mathbb{R}^n)$. If, additionally, $\sigma$ satisfies
    $$
        \lim_{(z,\zeta)\rightarrow\infty} V_{\varphi}\sigma(z,\zeta) = 0,
    $$
    then $Op_{\tau}(\sigma)$ is compact on $M^{p,q}_{m}(\mathbb{R}^n)$.
\end{thm}
\noindent Observe that Theorem \ref{TauPDOCompactness1} contains Theorem A and Theorem B in the case $\nu \equiv m \equiv 1$. As the $\nu$-localization of $Op_{\tau}(\sigma)$ for $\sigma \in M^{\infty,1}_{1\otimes \nu\circ J^{-1}}(\mathbb{R}^{2n})$ was established by Cordero, Nicola, and Trapasso in \cite{CNT2019}*{Theorem 4.1}, we deduce Theorem \ref{TauPDOCompactness1} from Theorem \ref{AbstractCompactnessTheorem} by verifying that the vanishing of $V_{\varphi}\sigma$ implies $Op_{\tau}(\sigma)$ is weakly compact.  

We also address the compactness of $\tau$-pseudodifferential operators with symbols in the Wiener amalgam space $W(\mathcal{F}L^{\infty},L^1_{\nu \circ \mathcal{B}_{\tau}})(\mathbb{R}^{2n})$ as follows. Recall that Wiener amalgam spaces are defined by reversing the roles of time and frequency in the definition of modulation spaces, and they can be interpreted as the image of modulation spaces under the Fourier transform. 
\begin{thm}\label{TauPDOCompactness2}
    Let $p,q \in [1,\infty)$, $\nu$ be an admissible weight, $m$ be a $\nu$-moderate weight, and $\tau \in (0,1)$. If $\sigma \in W(\mathcal{F}L^{\infty},L^1_{\nu \circ \mathcal{B}_{\tau}})(\mathbb{R}^{2n})$, then $Op_{\tau}(\sigma)$ is bounded from $M^{p,q}_{m\circ \mathcal{U}_{\tau}}(\mathbb{R}^n)$ to $M^{p,q}_m(\mathbb{R}^n)$. If, additionally, $\sigma$ satisfies
    $$
        \lim_{(z,\zeta)\rightarrow\infty} V_{\varphi}\sigma(z,\zeta) = 0,
    $$
    then $Op_{\tau}(\sigma)$ is compact from $M^{p,q}_{m\circ \mathcal{U}_{\tau}}(\mathbb{R}^n)$ to $M^{p,q}_{m}(\mathbb{R}^n)$.
\end{thm}
\noindent As \cite{CNT2019}*{Theorem 4.3} asserts that the operators of Theorem \ref{TauPDOCompactness2} are $(\nu,\mathcal{U}_{\tau})$-localized, we obtain Theorem \ref{TauPDOCompactness2} from Theorem \ref{AbstractCompactnessTheorem} by showing that $V_{\varphi}\sigma(z,\zeta)\rightarrow0$ implies $Op_{\tau}(\sigma)$ is $\mathcal{U}_{\tau}$-weakly compact.

\subsection{Three-parameter pseudodifferential operators}
Although the majority of research on pseudodifferential operators concerns $\tau$-pseudodifferential operators (especially $Op^{\text{KN}}(\sigma)$ and $Op^{\text{W}}(\sigma)$), Calder\'on and Vaillancourt's original work \cite{CV1972} actually included the more general three-parameter pseudodifferential operators $T_{\sigma}$ given for $\sigma:\mathbb{R}^{3n}\rightarrow \mathbb{C}$ by 
$$
    T_{\sigma}f(x) = \iint_{\mathbb{R}^{2n}} \sigma(x,y,\xi)e^{2\pi i(x-y)\cdot \xi}f(y)\,dyd\xi.
$$
In this generality, we show that if $\sigma$ is in a weighted Sj\"ostrand class, then $T_{\sigma}$ is $\nu$-localized, and deduce the following as a consequence of Theorem \ref{AbstractCompactnessTheorem}. 
\begin{thm}\label{GeneralPDOCompactness}
Let $p,q \in [1,\infty)$, $\nu$ be an admissible weight of polynomial growth, $m$ be a $\nu$-moderate weight, and $B(t_1,t_2,t_3) := (-t_3,t_1+t_2)$. If $\sigma \in M_{1 \otimes \nu\circ B}^{\infty,1}(\R^{3n})$, then $T_{\sigma}$ is bounded on $M^{p,q}_m(\mathbb{R}^n)$. If, additionally, $\sigma$ satisfies 
$$
    \lim_{z \rightarrow \infty} \sigma(\cdot-z_1, \cdot-z_1, \cdot-z_2) = 0
$$
in $\mathcal{S}'(\R^{3n})$, then $T_{\sigma}$ is compact on $M_m^{p,q}(\R^n)$.
\end{thm} 

Our strategy for proving Theorem \ref{GeneralPDOCompactness} involves first establishing an atomic decomposition for the weighted Sj\"ostrand class $M_{1 \otimes \nu \circ B}^{\infty,1}(\R^{3n})$ with atoms having Fourier transform supported near some $k \in \Z^{3n}$ (Theorem \ref{AtomicDecomp} below). 
We then show that operators with atomic symbols are $\nu$-localized and use the atomic decomposition to conclude the same for $T_{\sigma}$ with $\sigma \in M_{1 \otimes \nu \circ B}^{\infty,1}(\R^{3n})$. The result follows from Theorem \ref{AbstractCompactnessTheorem} upon exploiting the smoothness of atomic symbols to show that the vanishing of the translates of $\sigma$ in $\mathcal{S}'(\mathbb{R}^{3n})$ implies weak compactness, and lifting this fact to general symbols through the atomic decomposition.

\begin{rem}
    The vanishing of $\sigma$ imposed in Theorem \ref{GeneralPDOCompactness} is natural in the sense that, for two-parameter symbols $\sigma \in M^{\infty,1}(\R^{2n})$, we have $\lim_{(z,\zeta)\rightarrow \infty}V_{\varphi}\sigma(z,\zeta)=0$ if and only if $\lim_{(x,\xi) \rightarrow \infty} \sigma(\cdot-x,\cdot-\xi) = 0$ in $\mathcal{S}'(\R^{2n})$. Indeed, the reproducing formula \eqref{ReproducingFormula} and dominated convergence show that the vanishing of $V_{\varphi}\sigma$ implies that $\sigma(\cdot-x,\cdot-\xi) \rightarrow 0$ in $\mathcal{S}'(\R^{2n})$. Conversely, $\sigma(\cdot-x,\cdot-\xi) \rightarrow 0$ implies that $Op^{\text{KN}}(\sigma) = Op^{\text{W}}(S_{1\rightarrow \frac{1}{2}}\sigma)$ is weakly compact, and hence compact on $L^2(\mathbb{R}^n)$ by Theorem \ref{TauPDOCompactness1}. Thus $V_{\varphi}(S_{1\rightarrow\frac{1}{2}}\sigma)(z,\zeta) \rightarrow 0$ as $(z,\zeta) \rightarrow \infty$ by \cite{FG2007}*{Theorem 4.6}, which implies $V_{\varphi}\sigma(z,\zeta) \rightarrow 0$ since $S_{1\rightarrow\frac{1}{2}}$ is an isomorphism on $M^{0}(\mathbb{R}^{2n})$; see Definition \ref{M0Definition} and Definition \ref{TauTransform} below.
\end{rem}

\subsection{Fourier integral operators}

The Kohn-Nirenberg pseudodifferential operator is also the prototypical example in the class of Fourier integral operators. For a phase $\Phi:\mathbb{R}^{2n}\rightarrow\mathbb{R}$ and symbol $\sigma:\mathbb{R}^{2n}\rightarrow\mathbb{C}$, the Fourier integral operator $T_{\sigma,\Phi}$ is given by
$$
    T_{\sigma,\Phi}f(x) = \int_{\mathbb{R}^n} \sigma(x,\xi) e^{2\pi i \Phi(x,\xi)}\widehat{f}(\xi)\,d\xi.
$$

We obtain the following result as a consequence of Theorem \ref{AbstractCompactnessTheorem}. 
\begin{thm}\label{FIOCompactness}
    Let $p,q \in [1,\infty)$, $s\ge 0$, $S > s+2n$, and $m$ be a $\nu_s$-moderate weight. If $\Phi$ is a tame phase and $\sigma \in M^{\infty,1}_{1\otimes \nu_S}(\R^{2n})$, then $T_{\sigma,\Phi}$ is bounded from $M_{m,\chi}^{p,q}(\R^n)$ to $M_m^{p,q}(\R^n)$. If, additionally, $\sigma$ satisfies
    $$
        \lim_{(z,\zeta)\rightarrow\infty} V_{\varphi}\sigma(z,\zeta) = 0,
    $$
    then $T_{\sigma,\Phi}$ is compact from $M^{p,q}_{m,\chi}(\mathbb{R}^n)$ to $M^{p,q}_m(\mathbb{R}^n)$. 
\end{thm}

\noindent Note that Theorem \ref{FIOCompactness} recovers the $p<\infty$ case of \cite{FGP2019}*{Theorem 3.12} since $M_{m,\chi}^{p,p}(\R^n) = M_{m \circ \chi}^{p,p}(\R^n)$. Moreover, for general tame phases and $p \neq q$, $T_{\sigma,\Phi}$ might not even be bounded from $M_{m \circ \chi}^{p,q}(\R^n)$ to $M_m^{p,q}(\R^n)$; see \cite{CNR2010}*{Proposition 7.1} -- Theorem \ref{FIOCompactness} gives an appropriate replacement for boundedness (and compactness) in this off-diagonal case.

The paper is organized as follows. In Section \ref{AbstractCompactnessSection}, we provide the setup and prove Theorem \ref{AbstractCompactnessTheorem}. In Section \ref{ApplicationsSection}, we apply Theorem \ref{AbstractCompactnessTheorem} to Fourier integral operators by proving Theorem \ref{FIOCompactness}, to $\tau$-pseudodifferential operators by proving Theorem \ref{TauPDOCompactness1} and Theorem \ref{TauPDOCompactness2}, and to three-parameter pseudodifferential operators by proving Theorem \ref{GeneralPDOCompactness}.


\section{Compactness of localized operators on pullback modulation spaces}\label{AbstractCompactnessSection}

Let $\pi$ denote the time-frequency shift defined for $z = (x,\xi) \in \mathbb{R}^{2n}$ by
$$
    \pi(z) := M_{\xi} T_x,
$$
where $T_x$ and $M_{\xi}$ are the translation and modulation operators given for $f \in L^2(\R^n)$ by
$$
    T_xf(y) := f(y-x) \quad\text{and}\quad M_{\xi}f(y) := e^{2\pi i \xi \cdot y}f(y).
$$
We work with the space of test functions
\[
    {\mathcal{S}_{\vartheta}}(\mathbb{R}^{n}) := \{ f \in L^2(\R^n) : \|f\|_{\mathcal{S}_{\vartheta}}:=\|\langle f, \pi(\cdot)\varphi\rangle \|_{L^1_{\vartheta}(\mathbb{R}^{2n})}<\infty\},
\]
where $\varphi(x) := 2^{\frac{n}{4}}e^{-\pi|x|^2}$ and $\vartheta(z) := e^{|z|}$. Note that $\varphi \in \mathcal{S}_{\vartheta}(\mathbb{R}^n)$ and that $\varphi$ can be replaced with any nonzero window $g \in \mathcal{S}_{\vartheta}(\mathbb{R}^n)$ to yield the same space with an equivalent norm, see \cite{G2001}*{Proposition 11.4.2c}. 

Given $g \in \mathcal{S}_{\vartheta}(\mathbb{R}^n)$, the short-time Fourier transform of $f \in \mathcal{S}_{\vartheta}'(\mathbb{R}^n)$ is defined by 
$$
    V_gf(z) := \langle f,\pi(z)g\rangle.
$$
It is well known that the time-frequency shifts of the Gaussian window $\varphi$ form a Parseval frame for $L^2(\mathbb{R}^n)$ (called the Gabor frame) in the sense that
$$
    \|f\|_{L^2(\mathbb{R}^n)}^2 = \int_{\mathbb{R}^{2n}} |V_{\varphi}f(z)|^2\,dz
$$
for all $f \in L^2(\mathbb{R}^n)$. More generally, we have the following reproducing formula for any $g_1,g_2 \in \mathcal{S}_\vartheta(\mathbb{R}^n)$ with $\langle g_1,g_2 \rangle \neq 0$:
\begin{align}\label{ReproducingFormula}
    \langle f, h \rangle = \langle g_2,g_1 \rangle^{-1} \int_{\R^{2n}} \langle f, \pi(z) g_1 \rangle \langle \pi(z) g_2 , h \rangle \, dz
\end{align}
for all $f \in \mathcal{S}_\vartheta(\mathbb{R}^n)$ and $h \in \mathcal{S}_\vartheta'(\mathbb{R}^n)$. Indeed, $f = \langle g_1, g_1 \rangle^{-1}  \int_{\R^{2n}} \langle f, \pi(z) g_1 \rangle \pi(z)g_2 \,dz$ holds in the weak $L^2(\mathbb{R}^{2n})$ sense, and since the defining integral converges absolutely in $\mathcal{S}_\vartheta(\mathbb{R}^n)$ norm, we can interchange the action of a distribution. 

A weight is a positive measurable function. We say a weight $\nu$ on $\mathbb{R}^{2n}$ is admissible if 
$\nu(0)=1$, $\nu$ is continuous, $\nu$ is even in each coordinate:
$$
    \nu(\pm z_1,\ldots,\pm z_{2n}) = \nu(z_1,\ldots,z_{2n})
$$ 
for all $z = (z_1,\ldots, z_{2n}) \in \mathbb{R}^{2n}$, $\nu$ is submultiplicative:
$$
    \nu(z+w) \leq \nu(z)\nu(w)
$$ 
for all $z,w \in \mathbb{R}^{2n}$, and $\nu$ satisfies 
$$
    \lim_{n\rightarrow\infty} \nu(nz)^{\frac{1}{n}} = 1
$$
for all $z \in \mathbb{R}^{2n}$. 
We say that a weight $m$ is $\nu$-moderate if
$$
    m(z+w) \lesssim m(z)\nu(w)
$$
for all $z,w\in\mathbb{R}^{2n}$. For $s \ge 0$, we write $\nu_s(z) = (1+|z|^2)^{\frac{s}{2}}$. We say that a weight $\nu$ is of polynomial growth if $\nu(z) \lesssim \nu_s(z)$ for some $s \ge 0$ and all $z \in \mathbb{R}^{2n}$.

Given $p,q \in [1,\infty]$ and a weight $m$ on $\R^{2n}$, the mixed Lebesgue space $L^{p,q}_m(\mathbb{R}^{2n})$ is the space of $f:\mathbb{R}^{2n}\rightarrow\mathbb{C}$ such that
$$
    \|f\|_{L^{p,q}_m(\mathbb{R}^{2n})}:= \bigg(\int_{\mathbb{R}^n}\bigg(\int_{\mathbb{R}^n} |f(x,\xi)|^pm(x,\xi)^p\,dx\bigg)^{\frac{q}{p}}\,d\xi\bigg)^{\frac{1}{q}}<\infty
$$
with the usual interpretation when $p$ or $q$ equals $\infty$. We write $L^p_m(\mathbb{R}^{2n}) = L^{p,p}_m(\mathbb{R}^{2n})$ and omit the subscript $m$ when $m\equiv 1$.

\begin{defin}
    Given $p,q\in[1,\infty]$, $g \in {\mathcal{S}_{\vartheta}}(\mathbb{R}^n) \setminus \{0\}$, a bi-Lipschitz diffeomorphism $\chi : \R^{2n} \rightarrow \R^{2n}$, an admissible weight $\nu$, and a $\nu$-moderate weight $m$, the pullback modulation space $M_{m,\chi}^{p,q}(\R^n)$ is defined to be the space of all $f \in {\mathcal{S}_{\vartheta}}'(\mathbb{R}^n)$ such that
    \[
        \|f\|_{M^{p,q}_{m,\chi}(\mathbb{R}^n)}:= \|V_{g} f \circ \chi^{-1}\|_{L_m^{p,q}(\R^{2n})} < \infty.
    \]
    We write $M^{p}_{m,\chi}(\mathbb{R}^n) = M^{p,p}_{m,\chi}(\mathbb{R}^{n})$, omit the subscript $m$ when $m \equiv 1$, and omit the subscript $\chi$ when $\chi$ is the identity. 
\end{defin}
\begin{rem}\label{WeightComposition}
    If composition with $\chi$ is an isomorphism on $L^{p,q}(\R^{2n})$, then $M_{m,\chi}^{p,q}(\R^n) = M_{m \circ \chi}^{p,q}(\R^n)$. In particular, this condition is automatic when $p=q$ since $\chi$ is bi-Lipschitz.
\end{rem}

We observe that the space $M^{p,q}_{m,\chi}(\mathbb{R}^n)$ is independent of the choice of nonzero window $g \in\mathcal{S}_{\vartheta}(\mathbb{R}^n)$ and that different choices of window define equivalent norms. Moreover, $M^{p,q}_{m,\chi}(\mathbb{R}^n)$ is a Banach space that embeds continuously into $\mathcal{S}_{\vartheta}'(\mathbb{R}^n) = M^{\infty}_{\vartheta^{-1}}(\mathbb{R}^n)$ and contains $\mathcal{S}_{\vartheta}(\mathbb{R}^n) = M^1_{\vartheta}(\mathbb{R}^n)$ as a dense subset. 
We first record a couple of useful lemmata.
\begin{lemma}\label{LemmaOne}
    If $f, g \in \mathcal{S}_{\vartheta}(\mathbb{R}^n)\setminus\{0\}$ 
    and $z \in \mathbb{R}^{2n}$, then 
    $$
        |V_gf(z)| \lesssim \|f\|_{\mathcal{S}_{\vartheta}(\mathbb{R}^n)} \|g\|_{\mathcal{S}_{\vartheta}(\mathbb{R}^n)} \vartheta(z)^{-1}.
    $$
\end{lemma}
\begin{proof}
    By the reproducing formula \eqref{ReproducingFormula} and the continuous embedding ${\mathcal{S}_{\vartheta}}(\mathbb{R}^n) \hookrightarrow M_\vartheta^\infty(\mathbb{R}^n)$ of \cite{G2001}*{Theorem 12.1.9}, we have
    \begin{align*}
        |V_gf(z)|= |\langle f, \pi(z) g \rangle | &\leq \int_{\R^{2n}}  |\langle f, \pi(w) \varphi \rangle \langle \pi(w) \varphi , \pi(z) g \rangle | \, dw\\
        &=\int_{\R^{2n}}  |\langle f, \pi(w) \varphi \rangle| \vartheta(w) |\langle \pi(w-z) \varphi , g \rangle | \frac{\vartheta(w-z)}{\vartheta(w-z)}\vartheta(w)^{-1} \, dw\\
        &\leq \int_{\R^{2n}}  |\langle f, \pi(w) \varphi \rangle| \vartheta(w) |\langle \pi(w-z) \varphi , g \rangle | \vartheta(w-z)\vartheta(z)^{-1} \, dw\\
        &\leq \|f\|_{\mathcal{S}_{\vartheta}(\mathbb{R}^n)}\|g\|_{M_{\vartheta}^\infty(\mathbb{R}^n)} \vartheta(z)^{-1}\\
        &\lesssim \|f\|_{\mathcal{S}_{\vartheta}(\mathbb{R}^n)}\|g\|_{\mathcal{S}_{\vartheta}(\mathbb{R}^n)} \vartheta(z)^{-1}
    \end{align*}
    for any $z \in \mathbb{R}^{2n}$.
\end{proof}

\begin{lemma}\label{LemmaTwo}
    If $\nu$ is an admissible weight and $\delta >0$, then $\nu(z) \lesssim e^{\delta |z|}$ for all $z \in \mathbb{R}^{2n}$.
\end{lemma}
\begin{proof}
    Supposing otherwise, there exists an index $j \in \{ 1, ..., 2n\}$ and $\delta > 0$ such that $\nu(0, ..., z_j, ..., 0) e^{- \delta |z_j|}$ is unbounded. In particular, there is a sequence of positive numbers $\{s_k\}$ such that $s_k \rightarrow \infty$ and $\nu(0, ..., s_k, ..., 0)e^{- \delta s_k} \geq 1$. If $s_k' = \lfloor s_k \rfloor$, then
    \[
        \nu(0, ..., s_k', ..., 0)e^{- \delta s_k'} \geq c > 0
    \]
    so that
    \[
        1 = \lim_{k \rightarrow \infty} \nu(0, ..., s_k', ..., 0)^\frac{1}{s_k'} \geq e^\delta > 1,
    \]
    which contradicts the assumption that $\nu$ is admissible.
\end{proof}

\begin{prop}\label{ModulationSpaceProperties}
    If $p,q \in [1,\infty)$, $g \in \mathcal{S}_{\vartheta}(\mathbb{R}^n)\setminus \{0\}$, $\chi$ is a bi-Lipschitz diffeomorphism on $\mathbb{R}^{2n}$, $\nu$ is an admissible weight, and $m$ is a $\nu$-moderate weight, then
    \begin{enumerate}
        \item the inclusion $M_{m,\chi}^{p,q}(\R^n) \hookrightarrow {\mathcal{S}_{\vartheta}}'(\mathbb{R}^n)$ is continuous;
        \item $M_{m,\chi}^{p,q}(\R^n)$ is a Banach space;
        \item the space $M_{m,\chi}^{p,q}(\R^n)$ is independent of the choice of window $g \in \mathcal{S}_{\vartheta}(\mathbb{R}^n) \setminus \{0\}$ and different choices of window define equivalent norms; and
        \item $\mathcal{S}_\vartheta(\mathbb{R}^n)$ is dense in $M_{m,\chi}^{p,q}(\R^n)$.
    \end{enumerate}
\end{prop}

\begin{proof}
    We first prove (1). Assume without loss of generality that $\|g\|_{L^2(\mathbb{R}^n)} = 1$, let $h \in \mathcal{S}_{\vartheta}(\mathbb{R}^n)$, and let $f \in M_{m,\chi}^{p,q}(\R^n)$. By the reproducing formula \eqref{ReproducingFormula}, the fact that $\chi$ is bi-Lipschitz, H\"older's inequality for mixed Lebesgue spaces, and the fact that $\frac{1}{m}\lesssim \nu$ (which follows since $\nu$ is even and $m$ is $\nu$-moderate), we have
    \begin{align*}
        |\langle f, h \rangle| &= \left| \int_{\R^{2n}} \langle f, \pi(z) g \rangle \langle \pi(z)g , h \rangle \, dz \right|\\
        &\lesssim \int_{\R^{2n}} \left| \langle f, \pi(\chi^{-1}(z)) g \rangle \langle \pi(\chi^{-1}(z))g , h \rangle \right| \, dz\\
        &\leq \|f\|_{M_{m,\chi}^{p,q}(\R^n)} \| \langle \pi(\chi^{-1}(\cdot)) g , h \rangle\|_{L_{\frac{1}{m}}^{p',q'}(\R^{2n})}\\
        &\lesssim \|f\|_{M_{m,\chi}^{p,q}(\R^n)} \| \langle \pi(\chi^{-1}(\cdot)) g , h \rangle\|_{L_{\nu}^{p',q'}(\R^{2n})}.
    \end{align*}
    Since $\chi$ is bi-Lipschitz, there exist $C,\delta>0$ such that $|\chi^{-1}(z)| \geq \delta |z|-C$ for all $z \in \mathbb{R}^{2n}$. Lemma \ref{LemmaOne} gives $|\langle \pi(\chi^{-1}(z)) g , h \rangle| \lesssim \|g\|_{\mathcal{S}_\vartheta(\mathbb{R}^n)} \|h\|_{\mathcal{S}_{\vartheta}(\mathbb{R}^n)} e^{-\delta |z|}$, and so by Lemma \ref{LemmaTwo}, we have 
    $$
        |\langle \pi(\chi^{-1}(z)) g , h \rangle| \nu(z) \lesssim \|g\|_{\mathcal{S}_\vartheta(\mathbb{R}^n)} \|h\|_{\mathcal{S}_{\vartheta}(\mathbb{R}^n)} e^{-\frac{\delta}{2}|z|}.
    $$
    The $L^{p',q'}(\R^{2n})$-integrability of $e^{-\frac{\delta}{2}|z|}$ implies the needed estimate.
    
    To verify (2), we show that if a sequence $\{f_k\} \subseteq M^{p,q}_{m,\chi}(\mathbb{R}^n)$ satisfies $\sum_k \|f_k\|_{M_{m,\chi}^{p,q}(\R^n)} < \infty$, then $\{\sum_{k<N}f_k\}$ converges to some $f \in M^{p,q}_{m,\chi}(\mathbb{R}^{2n})$. With this assumption, part (1) gives that $\sum_k f_k$ converges in ${\mathcal{S}_{\vartheta}}'(\mathbb{R}^n)$ to some distribution $f$, which, by the convergence in ${\mathcal{S}_{\vartheta}}'(\mathbb{R}^n)$ and the countable subadditivity of $\|\cdot\|_{L_m^{p,q}(\R^{2n})}$, belongs to $M_{m,\chi}^{p,q}(\R^n)$. Furthermore,
    \begin{align*}
        \bigg\|f - \sum_{k < N} f_k \bigg\|_{M_{m,\chi}^{p,q}(\R^{n})} &= \bigg\| \sum_{k \geq N} \langle f_k, \pi(\chi^{-1}(\cdot)) g \rangle  \bigg\|_{L_m^{p,q}(\R^{2n})}\\
        &\leq \sum_{k \geq N} \| \langle f_k, \pi(\chi^{-1}(\cdot)) g \rangle  \|_{L_m^{p,q}(\R^{2n})}\\
        &= \sum_{k \geq N} \|f_k\|_{M_{m,\chi}^{p,q}(\R^{n})} \rightarrow 0
     \end{align*}
     as $N \rightarrow \infty$. This establishes (2).
     
     To show (3), suppose $g_0 \in {\mathcal{S}_{\vartheta}(\mathbb{R}^n)} \setminus \{0\}$ and denote 
     $$
        \widetilde{M_{m,\chi}^{p,q}}(\R^n) = \left\{ f \in {\mathcal{S}_{\vartheta}}'(\mathbb{R}^n) : \|f\|_{\widetilde{M_{m,\chi}^{p,q}}(\R^n)}:=\|V_{g_0}f \circ \chi^{-1}  \|_{L_m^{p,q}(\R^{2n})} < \infty \right\}.
    $$ 
    By the reproducing formula \eqref{ReproducingFormula} and since $\chi$ is bi-Lipschitz, we have
    \begin{align*}
        \|f\|_{\widetilde{M_{m,\chi}^{p,q}}(\R^n)} &= \bigg\| \int_{\R^{2n}} \langle f, \pi(w) g \rangle \langle \pi(w) g, \pi(\chi^{-1}(\cdot)) g_0  \rangle  \,dw \bigg\|_{L_m^{p,q}(\R^{2n})}\\
        &\lesssim \bigg\| \int_{\R^{2n}} \left| \langle f, \pi(\chi^{-1}(w)) g \rangle \langle g, \pi(\chi^{-1}(\cdot) - \chi^{-1}(w)) g_0  \rangle  \right| \, dw\bigg\|_{L_m^{p,q}(\R^{2n})}.
     \end{align*}
    Since $|\chi^{-1}(z) - \chi^{-1}(w)| \geq \delta |z-w|$ for some $\delta>0$, we have by Lemma \ref{LemmaOne} that 
    $$
        |\langle g, \pi(\chi^{-1}(z) - \chi^{-1}(w)) g_0 \rangle| \lesssim e^{-\delta|z-w|}.
    $$
    Since $e^{-\delta|z|} \in L_\nu^1$ and since $L_m^{p,q}(\R^{2n}) * L_\nu^1(\R^{2n}) \subseteq L_m^{p,q}(\R^{2n})$ (see \cite{G2001}*{Proposition 11.1.3 (a)}), we have that 
    $$
        \|f\|_{\widetilde{M_{m,\chi}^{p,q}}(\R^n)} \lesssim \|f\|_{M_{m,\chi}^{p,q}(\R^n)}.
    $$
    The reverse inequality holds by a symmetric argument.

    For (4), first note that $\mathcal{S}_\vartheta(\mathbb{R}^n) \hookrightarrow M_{m,\chi}^{p,q}(\R^n)$ by Lemma \ref{LemmaOne} and that $e^{-\delta |z|} \in L_m^{p,q}(\R^{2n})$. Now let $f \in M_{m,\chi}^{p,q}(\R^n)$. Let $\{K_k\}$ be an increasing compact exhaustion of $\R^{2n}$ and set
    \[
        f_k = \int_{K_k} \langle f, \pi(w) \varphi \rangle \pi(w) \varphi \, dw.
    \]
    Clearly, each $f_k$ is in $\mathcal{S}_\vartheta(\mathbb{R}^n)$. Also,
    \begin{align*}
    \|f - f_k\|_{M_{m,\chi}^{p,q}(\R^{n})} &\lesssim \bigg\|\int_{\R^{2n} \setminus \chi(K_k)} |\langle f, \pi(\chi^{-1}(w)) \varphi \rangle \langle \pi(\chi^{-1}(w)) \varphi , \pi(\chi^{-1}(\cdot)) \varphi \rangle |\, dw \bigg\|_{L_m^{p,q}(\R^{2n})}\\
    &\lesssim \bigg\|\int_{\R^{2n} \setminus \chi(K_k)} |\langle f, \pi(\chi^{-1}(w)) g \rangle| e^{-\delta |\cdot-w|} \, dw \bigg\|_{L_m^{p,q}(\R^{2n})} \rightarrow 0
    \end{align*}
    as $k \rightarrow \infty$ by the dominated convergence theorem, since $p,q < \infty$ and $L_m^{p,q} * L_\nu^1 \subseteq L_m^{p,q}$.
\end{proof}

\begin{rem}
    Note that (1)--(3) hold in the general case $p,q \in [1,\infty]$ with the same proof; however, (4) fails when $p$ or $q$ equals $\infty$. This technicality is one of the reasons we restrict to the range $p,q \in [1,\infty)$ in our main results.
\end{rem}

We next define the notion of a $(\nu,\chi,g_1,g_2)$-localized operator.
\begin{defin}\label{LocalizationDef}
    Given $g_1,g_2 \in \mathcal{S}_\vartheta(\mathbb{R}^n) \setminus \{ 0 \}$, a bi-Lipschitz diffeomorphism $\chi : \R^{2n} \rightarrow \R^{2n}$, and an admissible weight $\nu$, we say that $T : \mathcal{S}_{\vartheta}(\mathbb{R}^n) \rightarrow \mathcal{S}_{\vartheta}'(\mathbb{R}^n)$ is $(\nu,\chi,g_1,g_2)$-localized if 
    $$
        |\langle T \pi(z) g_1 , \pi(w)g_2 \rangle| \leq L(w - \chi(z))
    $$ 
    for some $L \in L_\nu^1(\R^{2n})$. We omit $\nu$, $\chi$, $g_1$, or $g_2$ from the notation when $\nu \equiv 1$, $\chi$ is the identity, or $g_j=\varphi$.
\end{defin}

We observe that $(\nu, \chi, g_1,g_2)$-localized operators are bounded from $M^{p,q}_{m,\chi}(\mathbb{R}^n)$ to $M^{p,q}_m(\mathbb{R}^n)$. 
\begin{prop}\label{localized>bounded}
    Let $p,q \in [1,\infty)$, $g_1,g_2 \in \mathcal{S}_{\vartheta}(\mathbb{R}^n) \setminus \{ 0 \}$, $\chi:\R^{2n}\rightarrow\R^{2n}$ be a bi-Lipschitz diffeomorphism, $\nu$ be an admissible weight, and $m$ be a $\nu$-moderate weight. If $T : \mathcal{S}_{\vartheta}(\mathbb{R}^n) \rightarrow \mathcal{S}_{\vartheta}'(\mathbb{R}^n)$ is $(\nu,\chi,g_1,g_2)$-localized, then $T$ is bounded from $M^{p,q}_{m,\chi}(\mathbb{R}^n)$ to $M_m^{p,q}(\R^n)$. 
\end{prop}
\begin{proof}
    Suppose $\|g_1\|_{L^2(\mathbb{R}^n)} = \|L\|_{L_\nu^1(\mathbb{R}^{2n})} = 1$ and let $f \in \mathcal{S}_\vartheta(\mathbb{R}^n)$. Use $L_m^{p,q} * L_\nu^1 \subseteq L_m^{p,q}$ to see  
    \begin{align*}
        \|Tf\|_{M^{p,q}_{m}(\mathbb{R}^n)} &= \| \langle Tf, \pi(\cdot) g_2 \rangle\|_{L_m^{p,q}(\R^{2n})}\\
        &= \bigg\| \int_{\R^{2n}} \langle f, \pi(w) g_1 \rangle \langle T \pi(w) g_1 , \pi(\cdot) g_2 \rangle \, dw \bigg\|_{L_m^{p,q}(\R^{2n})}\\
        &\lesssim \bigg\| \int_{\R^{2n}} \left| \langle f, \pi(\chi^{-1}(w)) g_1 \rangle \langle T \pi(\chi^{-1}(w)) g_1 , \pi(\cdot) g_2 \rangle \right| \, dw \bigg\|_{L_m^{p,q}(\R^{2n})}\\
        &\leq \bigg\|\int_{\R^{2n}} \left| \langle f, \pi(\chi^{-1}(w)) g_1 \rangle L(\cdot-w) \right| \, dw \bigg\|_{L_m^{p,q}(\R^{2n})}\\
        &\lesssim \|f\|_{M_{m,\chi}^{p,q}(\R^n)}.
    \end{align*}
    The result follows by the density of $\mathcal{S}_{\vartheta}(\mathbb{R}^n)$ in $M^{p,q}_{m,\chi}(\mathbb{R}^n)$ from Proposition \ref{ModulationSpaceProperties} part (4).
\end{proof}

We next define the notion of a $(\chi,g_1,g_2)$-weakly compact operator. 
\begin{defin}
    Given $g_1,g_2 \in \mathcal{S}_{\vartheta}(\mathbb{R}^n)\setminus \{0\}$ and a bi-Lipschitz diffeomorphism $\chi: \R^{2n}\rightarrow\R^{2n}$, we say that $T : \mathcal{S}_{\vartheta}(\mathbb{R}^n) \rightarrow \mathcal{S}_{\vartheta}'(\mathbb{R}^n)$ is $(\chi,g_1,g_2)$-weakly compact if for every compact set $K\subseteq\mathbb{R}^{2n}$, we have 
    $$
        \lim_{z\rightarrow \infty}\sup_{w \in z+K} |\langle T \pi(\chi^{-1}(w)) g_1 , \pi(z) g_2 \rangle|=0.
    $$
    We omit $\chi$, $g_1$, or $g_2$ from the notation when $\chi$ is the identity or $g_j = \varphi$. 
\end{defin}

\begin{prop}\label{WeakCpt>Cpt}
    Let $p,q \in [1,\infty)$, $g_1,g_2 \in \mathcal{S}_{\vartheta}(\mathbb{R}^n) \setminus \{ 0 \}$, $\chi:\R^{2n}\rightarrow\R^{2n}$ be a bi-Lipschitz diffeomorphism, $\nu$ be an admissible weight, and $m$ be a $\nu$-moderate weight. If $T : \mathcal{S}_{\vartheta}(\mathbb{R}^n) \rightarrow \mathcal{S}_{\vartheta}'(\mathbb{R}^n)$ is $(\nu,\chi,g_1,g_2)$-localized and $(\chi,g_1,g_2)$-weakly compact, then $T$ is compact from $M^{p,q}_{m,\chi}(\mathbb{R}^n)$ to $M_m^{p,q}(\R^n)$.
\end{prop}
\begin{proof}
    Assume that $\|g_1\|_{L^2(\mathbb{R}^n)} = 1$. Appealing to \cite{DFG2002}*{Theorem 5}, it suffices to verify that for every $\epsilon>0$ there exists a compact set $K\subseteq\mathbb{R}^{2n}$ such that 
    $$
        \sup_{\substack{f \in M^{p,q}_{m,\chi}(\mathbb{R}^n) \\ \|f\|_{M_{m,\chi}^{p,q}(\R^n) \leq 1}}} \|\mathbbm{1}_{K^c}V_{g_2}(Tf)\|_{L^{p,q}_m(\mathbb{R}^{2n})} <\epsilon.
    $$
    To this end, let $\epsilon>0$ and choose a symmetric compact set $K'\subseteq\mathbb{R}^{2n}$ 
    such that 
    \[
        \sup_{\substack{h \in L^{p,q}_m(\mathbb{R}^{2n})\\\|h\|_{L_m^{p,q}(\R^{2n})}} \leq 1} \|h * (\mathbbm{1}_{(K')^c}L)\|_{L^{p,q}_m(\mathbb{R}^{2n})} < \epsilon,
    \]
    where $L \in L^1_{\nu}(\mathbb{R}^{2n})$ is from Definition \ref{LocalizationDef}, which exists by dominated convergence since  
    $$
        \|h * (\mathbbm{1}_{(K')^c}L)\|_{L^{p,q}_m(\mathbb{R}^{2n})} \lesssim \|h\|_{L^{p,q}_m(\mathbb{R}^{2n})} \|\mathbbm{1}_{(K')^c}L\|_{L^1_{\nu}(\mathbb{R}^{2n})}.
    $$
    By weak compactness, we can find a compact set $K \subseteq \mathbb{R}^{2n}$ such that
    $$
        \sup_{z \in \mathbb{R}^{2n}\setminus K}\sup_{w \in z+K'} |\langle T \pi(\chi^{-1}(w)) g_1 , \pi(z) g_2 \rangle | < \frac{\epsilon}{\nu(K')}.
    $$

    Note that by the reproducing formula \eqref{ReproducingFormula}, we have  
    \begin{align*}
        \mathbbm{1}_{K^c}(z) |V_{g_2}(Tf)(z)| &= \mathbbm{1}_{K^c}(z)  \left| \int_{\R^{2n}} \langle f, \pi(w) g_1 \rangle \langle T \pi(w) g_1 , \pi(z) g_2 \rangle \, dw \right|\\
        &\lesssim \mathbbm{1}_{K^c}(z)  \int_{\R^{2n}} | V_{g_1}f(\chi^{-1}(w)) \langle T \pi(\chi^{-1}(w)) g_1, \pi(z) g_2 \rangle | \, dw
    \end{align*}
    for any $z \in \mathbb{R}^{2n}$, and so we can estimate 
    \begin{align*}
        \|\mathbbm{1}_{K^c}V_{g_2}(Tf)\|_{L^{p,q}_m(\mathbb{R}^{2n})} &\lesssim \bigg\|\mathbbm{1}_{K^c} \int_{\mathbb{R}^{2n}} |V_{g_1}f(\chi^{-1}(w))||\langle T\pi(\chi^{-1}(w))g_1,\pi(\cdot)g_2\rangle|\,dw\bigg\|_{L^{p,q}_m(\mathbb{R}^{2n})}\\
        &\leq \bigg\|\mathbbm{1}_{K^c} \int_{(\cdot + K')^c} |V_{g_1}f(\chi^{-1}(w))||\langle T\pi(\chi^{-1}(w))g_1,\pi(\cdot)g_2\rangle|\,dw\bigg\|_{L^{p,q}_m(\mathbb{R}^{2n})}\\
        &\quad\quad+\bigg\|\mathbbm{1}_{K^c} \int_{(\cdot +K')} |V_{g_1}f(\chi^{-1}(w))||\langle T\pi(\chi^{-1}(w))g_1,\pi(\cdot)g_2\rangle|\,dw\bigg\|_{L^{p,q}_m(\mathbb{R}^{2n})}.
    \end{align*}
    We bound the first term by localization and the choice of $K'$: 
    \begin{align*}
        \bigg\|\mathbbm{1}_{K^c} \int_{(\cdot + K')^c} &|V_{g_1}f(\chi^{-1}(w))||\langle T\pi(\chi^{-1}(w))g_1,\pi(\cdot)g_2\rangle|\,dw\bigg\|_{L^{p,q}_m(\mathbb{R}^{2n})} \\
        &\quad \leq \bigg\|\int_{(\cdot + K')^c} |V_{g_1}f(\chi^{-1}(w))|L(\cdot -w)\,dw\bigg\|_{L^{p,q}_m(\mathbb{R}^{2n})}\\
        &\quad = \| |V_{g_1}f\circ\chi^{-1}|*(\mathbbm{1}_{(K')^c}L)\|_{L^{p,q}_m(\mathbb{R}^{2n})}\\
        &\quad < \epsilon.
    \end{align*}
    The second term is controlled by the choice of $K$ as follows:
    \begin{align*}
        \bigg\|\mathbbm{1}_{K^c} \int_{(\cdot +K')} &|V_{g_1}f(\chi^{-1}(w))||\langle T\pi(\chi^{-1}(w))g_1,\pi(\cdot)g_2\rangle|\,dw\bigg\|_{L^{p,q}_m(\mathbb{R}^{2n})}\\
        &\quad \leq \frac{\epsilon}{\nu(K')}\bigg\| \int_{(\cdot + K')} |V_{g_1}f(\chi^{-1}(w))|\,dw\bigg\|_{L^{p,q}_m(\mathbb{R}^{2n})}\\
        &\quad = \frac{\epsilon}{\nu(K')} \||V_{g_1}f\circ \chi^{-1}|*\mathbbm{1}_{K'}\|_{L^{p,q}_m(\mathbb{R}^{2n})}\\
        &\quad \lesssim \frac{\epsilon}{\nu(K')}\|V_{g_1}f\circ \chi^{-1}\|_{L^{p,q}_m(\mathbb{R}^{2n})}\|\mathbbm{1}_{K'}\|_{L^1_{\nu}(\mathbb{R}^{2n})}\\
        &\quad < \epsilon.
    \end{align*}
    The result follows since these estimates hold uniformly over all $\|f\|_{M^{p,q}_{m,\chi}(\mathbb{R}^n)}\leq 1$.
\end{proof}

\begin{proof}[Proof of Theorem \ref{AbstractCompactnessTheorem}]
This follows from Proposition \ref{localized>bounded} and Proposition \ref{WeakCpt>Cpt}.
\end{proof}

\section{Applications}\label{ApplicationsSection}

The sufficient conditions for our compactness results are phrased in terms of membership of the symbol $\sigma$ in $M^0(\mathbb{R}^n)$, which is defined as follows.
\begin{defin}\label{M0Definition}
    The space $M^0(\R^n)$ is defined to be space of all $f \in \mathcal{S}'(\mathbb{R}^n)$ such that 
    $$
        \lim_{z\rightarrow \infty}V_{\varphi}f(z) = 0.
    $$
\end{defin}
\noindent Note that $M^0(\mathbb{R}^n)$ is the $M^\infty(\R^n)$ closure of the Schwartz class $\mathcal{S}(\R^n)$; see \cite{BGHO2005}*{Lemma 2.2}.


\subsection{Fourier integral operators}\label{FIOSubsection}

In this subsection, we study the compactness of Fourier integral operators. We say that $\Phi$ is a tame phase if it is smooth with $\partial^\alpha \Phi \in L^{\infty}(\mathbb{R}^{2n})$ for all $|\alpha| \geq 2$ and if $|\det D^2 \Phi| \geq \delta > 0$. 
The canonical transformation of $\Phi$ is defined to be the function $\chi : \R^{2n} \rightarrow \R^{2n}$ defined by $(x,\xi) = \chi(y,\eta)$ and the system
\[
    \begin{cases}
        y = \nabla_\eta \Phi(x,\eta) \\
        \xi = \nabla_x \Phi(x,\eta)
    \end{cases}.
\]
Note that $\chi$ is a bi-Lipschitz diffeomorphism of $\R^{2n}$; see \cite{CGN2012}.

We establish the following uniform vanishing estimate for the matrix coefficients of a Fourier integral operator with tame phase and $M^0(\mathbb{R}^{2n})$ symbol.
\begin{prop}\label{FIOWBPestimate}
    If $\sigma \in M^0(\R^{2n})$ and $\Phi$ is a tame phase, then 
    \[
        \lim_{z \rightarrow \infty} \sup_{w \in \R^{2n}} |\langle T_{\sigma, \Phi} \pi(z) \varphi, \pi(w) \varphi \rangle| = 0.
    \]
\end{prop}
\begin{proof}
We first claim that if $\sigma \in M^{\infty}(\R^{2n})$, then
    \begin{align}\label{FIOClaim}
        \sup_{z,w \in \R^{2n}} |\langle T_{\sigma, \Phi} \pi(z) \varphi , \pi(w) \varphi \rangle| \lesssim \|\sigma\|_{M^\infty(\R^{2n})}.
    \end{align}
    To prove this, we first estimate 
\begin{align*}
    |\langle T_{\sigma,\Phi} \pi(z) \varphi, &\pi(w) \varphi \rangle| = \left| \iint_{\mathbb{R}^{2n}} \sigma(x,\xi) e^{2\pi i \Phi(x,\xi)} e^{-2 \pi i \xi \cdot z_1} \widehat{\varphi}(\xi - z_2) e^{-2 \pi i x \cdot w_2} \overline{\varphi}(x-w_1)  \,d\xi dx \right|\\
    &= \left| \iint_{\mathbb{R}^{2n}} \sigma(x+w_1,\xi+z_2) e^{2\pi i \Phi(x+w_1,\xi+z_2)} e^{-2 \pi i \xi \cdot z_1} \varphi(\xi) e^{-2 \pi i x \cdot w_2} \varphi(x)  \,d\xi dx \right|\\
    &\leq \|\sigma\|_{M^{\infty}(\R^{2n})} \| e^{2 \pi i \Phi(x+w_1, \xi + z_2)} \varphi(\xi) \varphi(x) \|_{M^{1}(\R^{2n})}.
\end{align*}
where we used the translation invariance of $M^\infty(\R^{2n})$ and modulation invariance of $M^1(\R^{2n})$. 

We verify \eqref{FIOClaim} by showing that the second factor above is bounded uniformly in $z$ and $w$. Note that it suffices to obtain this bound with $\Phi(x,\xi)$ in place of $\Phi(x+w_1, \xi+z_2)$, as long as the constants depend only on $\|\partial^\alpha \Phi\|_{L^{\infty}(\mathbb{R}^{2n})}$ for $|\alpha| \geq 2$. To this end, let $\theta \in C_0^\infty(\R^{2n})$ satisfy $\sum_{k \in \Z^{2n}} \theta(\cdot-k)\equiv 1$ and write $G(x,\xi) := \varphi(\xi) \varphi(x)$. Then
\[
    \| e^{2 \pi i \Phi(x, \xi)} G(x,\xi) \|_{M^{1}(\R^{2n})} \leq \sum_{k \in \Z^{2n}} \|\theta((x,\xi) - k)e^{2 \pi i \Phi(x,\xi)}  G(x,\xi)\|_{M^{1}(\R^{2n})}.
\]
For $(x,\xi)$ in the support of $\theta(\cdot -k)$, the smoothness of $\Phi$ and Taylor's theorem allow us to write $\Phi(x,\xi) = c_k + c_k' \cdot(x,\xi) + \Psi_k(x,\xi)$, where $c_k,c_k' \in \mathbb{R},\R^{2n}$ and $\Psi_k$ is smooth with derivatives that are bounded uniformly over $k \in \mathbb{Z}^{2n}$. Modulation invariance of $M^{1}(\mathbb{R}^{2n})$ then gives
\[
    \|\theta((x,\xi) - k)e^{2 \pi i \Phi(x,\xi)} G(x,\xi)\|_{M^{1}(\R^{2n})} = \| \theta((x,\xi)-k) e^{2 \pi i \Psi_k(x,\xi)} G(x,\xi) \|_{M^{1}(\R^{2n})}.
\]
Since $e^{2 \pi i \Psi_k(x,\xi)} G(x,\xi)$ is a Schwartz function with seminorms bounded uniformly in $k$, $\theta((x,\xi)-k) e^{2 \pi i \Psi_k(x,\xi)} G(x,\xi)$ is a Schwartz function with seminorms controlled by a constant times $(1+|k|)^{-2n-1}$. Since $\mathcal{S}(\R^{2n})$ continuously embeds into $M^1(\mathbb{R}^{2n})$, we have
\[
    \sum_{k \in \Z^{2n}} \|\theta((x,\xi) - k)e^{2 \pi i \Phi(x,\xi)} G(x,\xi)\|_{M^{1}(\R^{2n})} \lesssim \sum_{k \in \mathbb{Z}^{2n}} \frac{1}{(1+|k|)^{2n+1}}\lesssim 1,
\]
which establishes \eqref{FIOClaim}.

Note that if $\sigma \in \mathcal{S}(\R^{2n})$, then $T_{\sigma,\Phi}$ is $L^2(\mathbb{R}^n)$-compact (in fact, it is Hilbert-Schmidt as it is the composition with the Fourier transform of an integral operator having a square-integrable kernel). This implies that $\|T_{\sigma,\Phi} \pi(z) \varphi \|_{L^2(\mathbb{R}^n)} \rightarrow 0$ as $z \rightarrow \infty$ because $\pi(z) \varphi \rightarrow 0$ weakly in $L^2(\mathbb{R}^n)$, and so the result holds by Cauchy-Schwarz.  The result for general $\sigma \in M^0(\mathbb{R}^{2n})$ follows by density of $\mathcal{S}(\mathbb{R}^{2n})$ in $M^0(\mathbb{R}^{2n})$ with respect to the $M^{\infty}(\mathbb{R}^{2n})$ norm and \eqref{FIOClaim}.
\end{proof}

\begin{proof}[Proof of Theorem \ref{FIOCompactness}]
    The result follows from Theorem \ref{AbstractCompactnessTheorem} since $T_{\sigma,\Phi}$ is $(\nu,\chi)$-localized by \cite{CGN2012}*{Theorem 3.3} and $\chi$-weakly compact by Proposition \ref{FIOWBPestimate}.
\end{proof}


\subsection{\boldmath${\tau}$\unboldmath-pseudodifferential operators}\label{OptauSubsection}

We use the following transformation to relate $\tau$-pseudodifferential operators for different choices of $\tau$.
\begin{defin}\label{TauTransform}
    For $\tau_1, \tau_2 \in [0,1]$ with $\tau_1 \neq \tau_2$, we define
    \[
        S_{\tau_1 \rightarrow \tau_2} \sigma = |\tau_1-\tau_2|^{-n} e^{2 \pi i (\tau_2 - \tau_1) x \cdot \xi}*\sigma.
    \]
\end{defin}
\noindent Note that $S_{\tau_1 \rightarrow \tau_2}$ is an isomorphism of $M^{p,q}(\R^{2n})$ for $1 \leq p,q \leq \infty$, $S_{\tau_1 \rightarrow \tau_2}^{-1} = S_{\tau_2 \rightarrow \tau_1}$, and
\[
    Op_{\tau_1} (\sigma) = Op_{\tau_2} (S_{\tau_1 \rightarrow \tau_2} \sigma);
\]
see \cite{CDT2019}*{Section 5}. Further, since $S_{\tau_1 \rightarrow \tau_2}$ preserves Schwartz functions (and hence preserves $M^0(\mathbb{R}^{2n})$) and since $S_{\tau_1 \rightarrow \tau_2}^{-1} = S_{\tau_2 \rightarrow \tau_1}$, we see that $S_{\tau_1 \rightarrow \tau_2}$ is an isomorphism of $M^0(\R^{2n})$.

\begin{prop}\label{STFT>WeakCpt}
    If $\tau \in [0,1]$ and $\sigma \in M^0(\R^{2n})$, then
    \[
        \lim_{z \rightarrow \infty} \sup_{w \in \R^{2n}} |\langle Op_{\tau} (\sigma) \pi(z) \varphi, \pi(w) \varphi \rangle| = 0.
    \]
\end{prop}
\begin{proof}
    The case $\tau = 1$ holds by applying Proposition \ref{FIOWBPestimate} with $\Phi(x,\xi) = x \cdot \xi$. The conclusion for $\tau \in [0,1)$ follows from the case $\tau =1$ using the facts that  
    \[
        Op_\tau (\sigma) = Op_1 (S_{\tau \rightarrow 1} \sigma).
    \]
    and $S_{\tau\rightarrow 1}$ is an isomorphism of $M^0(\mathbb{R}^n)$.
\end{proof}

Following \cite{CNT2019}, we define the matrices and spaces appearing in Theorems \ref{TauPDOCompactness1} and \ref{TauPDOCompactness2}. Let
    $$
        J := \left( \begin{matrix}
        0 & Id_{n \times n}\\
        -Id_{n \times n} & 0
    \end{matrix} \right),
    $$
    $$
        \mathcal{B}_\tau := \left(\begin{matrix}
        \frac{1}{1-\tau} Id_{n \times n} & 0 \\
        0 & \frac{1}{\tau} Id_{n \times n}
    \end{matrix} \right),
    $$
    and
    $$
    \mathcal{U}_\tau := -\left(\begin{matrix}
        \frac{\tau}{1-\tau} Id_{n \times n} & 0 \\
        0 & \frac{1-\tau}{\tau} Id_{n \times n}
    \end{matrix} \right).
    $$
    For weights $\mu$ and $\nu$, we define the Wiener amalgam space to be 
    $$
        W(\mathcal{F} L_{\mu}^p, L_{\nu}^q )(\R^{n}) := \widehat{M_{\mu \otimes \nu}^{p,q}(\R^n)};
    $$
    in particular, 
    \[
        W(\mathcal{F} L^\infty, L_{\nu \circ\mathcal{B}_\tau}^1)(\R^{2n}) = \widehat{M_{1 \otimes \nu \circ \mathcal{B}_\tau}^{\infty,1}(\R^{2n})}.
    \]

\begin{proof}[Proof of Theorem \ref{TauPDOCompactness1}]
    We have that $Op_{\tau}(\sigma)$ is $\nu$-localized by applying \cite{CNT2019}*{Theorem 4.1} and noting that $\varphi \in  M_\nu^1$. The compactness is therefore a consequence of Theorem \ref{AbstractCompactnessTheorem} as Proposition \ref{STFT>WeakCpt} implies that $Op_{\tau}(\sigma)$ is weakly compact.
\end{proof}

\begin{proof}[Proof of Theorem \ref{TauPDOCompactness2}]
    We have that $Op_{\tau}(\sigma)$ is $(\nu,\mathcal{U}_{\tau})$-localized by \cite{CNT2019}*{Theorem 4.3}. For the compactness, first note that composition with $\mathcal{U}_\tau$ is an isomorphism on $L^{p,q}(\R^{2n})$, 
    so $M_{m,\mathcal{U}_\tau}^{p,q}(\R^n) = M_{m \circ\mathcal{U}_\tau}^{p,q}(\R^n)$. The result follows from Theorem \ref{AbstractCompactnessTheorem} and Proposition \ref{STFT>WeakCpt}.
\end{proof}


\subsection{Pseudodifferential operators with three-parameter symbols}

We prove our main result for pseudodifferential operators with three-parameter symbols in a certain weighted Sj\"ostrand class, Theorem \ref{GeneralPDOCompactness}, through an atomic decomposition of the symbol $\sigma$ into atoms $\sigma_k$ whose Fourier transforms are supported near $k \in \Z^{3n}$. We first investigate the pseudodifferential operators $T_{\sigma}$ associated with such atoms -- throughout this subsection, we consider bounded symbols $\sigma:\mathbb{R}^{3n}\rightarrow\mathbb{C}$ with $\text{supp} \, \widehat{\sigma} \subseteq K$ for some compact $K \subseteq \mathbb{R}^{3n}$. It follows that $\|\partial^\alpha \sigma\|_{L^{\infty}(\mathbb{R}^{3n})} \leq C(n,K,\alpha) \|\sigma\|_{L^{\infty}(\mathbb{R}^{3n})}$ for all multi-indices $\alpha$.

\begin{lemma}\label{DerivativeBound}
    Let $\sigma\in L^{\infty}(\mathbb{R}^{3n})$ with $\text{supp}\,\widehat{\sigma} \subseteq K$ for some compact $K \subseteq \mathbb{R}^{3n}$. If $\alpha$ is a multi-index and $N \in \N$, then there exists $C(n,K,\alpha,N)\ge 0$ such that
    $$
        |\partial_x^\alpha T_\sigma \varphi (x)| \leq C(n,K,\alpha,N) \|\sigma\|_{L^\infty} \langle x \rangle^{-N}
    $$
    for every $x \in \mathbb{R}^n$, where $\langle x \rangle := \left( 1 + |x|^2 \right)^\frac{1}{2}$.
\end{lemma}
    \begin{proof}
        Assume that $N \geq \max \{ \lfloor \frac{n}{2}\rfloor +1, \frac{|\alpha |}{2}\}$ and $\|\sigma\|_{L^{\infty}(\mathbb{R}^{3n})} = 1$. By the Leibniz rule and since $(I-\Delta_{\xi})^Ne^{2\pi i(x-y)\cdot \xi} = (1+4\pi^2|x-y|^2)^Ne^{2\pi i(x-y)\cdot\xi}$, we have that $\partial_x^\alpha T_\sigma \varphi (x)$ is equal to 
        $$
            \sum_{\zeta \leq \alpha} \binom{\alpha}{\zeta} \iint_{\mathbb{R}^{2n}} \frac{\varphi(y)}{(1+4 \pi ^2 |x-y|^2)^N} (2 \pi i \xi)^{\alpha - \zeta}\partial_x^\zeta \sigma(x,y,\xi)  \left( I - \Delta_\xi \right)^N e^{2\pi i (x-y)\cdot\xi}\,dyd\xi.
        $$
        Fix $\zeta \leq \alpha$ and put $\eta(x,y,\xi) = (2 \pi i \xi)^{\alpha - \zeta} \partial_x^\zeta \sigma(x,y,\xi)$. Estimate the above terms using integration by parts twice and the fact that $(I-\Delta_{y})^{2N}e^{2\pi i(x-y)\cdot \xi} = (1+4\pi^2|\xi|^2)^{2N}e^{2\pi i(x-y)\cdot\xi}$: 
        \begin{align*}
            &\iint_{\mathbb{R}^{2n}} \frac{\varphi(y)}{(1+4 \pi ^2 |x-y|^2)^N} e^{2\pi i (x-y)\cdot\xi}  \left(I - \Delta_\xi \right)^N \eta(x,y,\xi)  \,dyd\xi\\
            &\quad\quad =\iint_{\mathbb{R}^{2n}} \frac{\varphi(y)}{(1+4 \pi ^2 |x-y|^2)^N} \left( \frac{\left( I - \Delta_y \right)^{2N}e^{2\pi i (x-y)\cdot\xi}}{(1 + 4 \pi^2 |\xi|^2)^{2N}}\right)  \left(I - \Delta_\xi \right)^N \eta(x,y,\xi)  \,dyd\xi\\
            &\quad\quad = \iint_{\mathbb{R}^{2n}} \frac{e^{2\pi i (x-y)\cdot\xi}}{(1 + 4 \pi^2 |\xi|^2)^{2N}}  \left( I - \Delta_y \right)^{2N} \left( \frac{\varphi(y)}{(1+4 \pi ^2 |x-y|^2)^N }  \left(I - \Delta_\xi \right)^N \eta(x,y,\xi)  \right) \,dyd\xi.
        \end{align*}
        Expanding $(I-\Delta_{\xi})^N$ and $(I-\Delta_y)^{2N}$ and applying the Leibniz rule, it is enough to estimate
        \[
            \iint_{\mathbb{R}^{2n}} \frac{e^{2\pi i (x-y)\cdot\xi}}{ (1 + 4 \pi^2 |\xi|^2)^{2N}}   \partial_y^{\mathcal{I}}\left( \frac{\varphi(y)}{(1+4 \pi ^2 |x-y|^2)^N} \right)  \partial_\xi^{\mathcal{K}} \partial_y^{\mathcal{J}}\eta(x,y,\xi)  \,dyd\xi, 
        \]
        where $|\mathcal{I}|, |\mathcal{J}| \leq 4N$ and $|\mathcal{K}| \leq 2N$. Applying the Leibniz rule and using the fact that the derivatives of $(1+4\pi^2|x-y|^2)^{-N}$ decay at worst like $(1+|x-y|^2)^{-N}$, we have
        $$
            \partial_y^{\mathcal{I}} \left(\frac{\varphi(y)}{(1+4 \pi ^2 |x-y|^2)^N}\right) \lesssim \frac{1}{(1+|x-y|^2)^N (1+|y|^2)^{2N}},
        $$
        and since $\sigma$ has bounded derivatives, we have
        $$
            \partial_\xi^\mathcal{K} \partial_y^\mathcal{J} \eta(x,y,\xi) \lesssim  |\xi|^{|\alpha|} \leq (1+|\xi|^2)^N
        $$ 
        where we used $N \geq |\alpha|/2$. Therefore, we can control the double integral above by
        \begin{align*}
            &\iint_{\mathbb{R}^{2n}} \frac{1}{(1 + |\xi|^2)^{N}(1+|x-y|^2)^N(1+|y|^2)^{2N}} \,dyd\xi \\
            &\quad\quad\approx  \int_{\mathbb{R}^n} \frac{1}{(1+|x-y|^2)^N(1+|y|^2)^{2N}}   \,dy\\
            &\quad\quad\lesssim \frac{1}{(1+|x|^2)^N}\int_{\mathbb{R}^n}\frac{1}{(1+|y|^2)^N}\,dy \approx \frac{1}{(1+|x|^2)^N} \leq \langle x \rangle^{-N},
        \end{align*}
        and the result holds.
    \end{proof}

\begin{lemma}\label{PairingBound}
    Let $\sigma\in L^{\infty}(\mathbb{R}^{3n})$ with $\text{supp}\,\widehat{\sigma} \subseteq K$ for some compact $K \subseteq \mathbb{R}^{3n}$. If $N \in \N$, then there exists $C(n,K,N)>0$ such that 
    \[
        \sup_{w \in \R^{2n}} |\langle \pi(w) T_\sigma \pi(w)^* \varphi, \pi(z) \varphi \rangle| \leq C(n,K,N)\|\sigma\|_{L^{\infty}(\mathbb{R}^{3n})} \langle z \rangle^{-N}
    \]
    for all $z \in \mathbb{R}^{2n}$. 
    \end{lemma}
    \begin{proof}
        Assume $N \geq n+1$ and $\|\sigma\|_{L^{\infty}(\mathbb{R}^{3n})}=1$. Since $\pi(w) T_{\sigma} \pi(w)^*$ is a pseudodifferential operator associated with a symbol $\sigma'(x,y,\xi) = \sigma(x-w_1,y-w_1,\xi-w_2)$ such that $\|\sigma'\|_{L^\infty(\mathbb{R}^{3n})} = \|\sigma\|_{L^\infty(\mathbb{R}^{3n})}$ and $\text{supp} \, \widehat{\sigma'} = \text{supp} \, \widehat{\sigma}$, it suffices to assume $w=0$. Writing $(I-\Delta_x)^Ne^{-2\pi ix\cdot z_2} = (1+4\pi^2|z_2|^2)^Ne^{-2\pi ix\cdot z_2}$, expanding $(I-\Delta_x)^N$, and using integration by parts, we have
        \begin{align*}
            |\langle T_\sigma \varphi , \pi(z) \varphi \rangle| &= (1+4\pi^2|z_2|^2)^{-N} \left|\int_{\R^n} T_\sigma \varphi(x) \varphi(x-z_1) \left( I - \Delta_x \right)^N e^{-2\pi i x \cdot z_2}\,dx\right|\\
            &\lesssim (1+|z_2|^2)^{-N}\sum_{|\alpha|,|\beta| \leq 2N} \int_{\R^n} \left| \partial_x^\alpha T_\sigma \varphi(x) \partial_x^\beta \varphi(x-z_1) \right| \, dx\\
            &= (1+|z_2|^2)^{-N}\sum_{|\alpha|,|\beta| \leq 2N} \int_{\R^n} \left| \partial^\alpha T_\sigma \varphi(z_1-x) \partial^\beta \varphi(-x) \right| \, dx.
        \end{align*}
        Lemma \ref{DerivativeBound} gives a bound on each above integral by a constant times
        \[
            \int_{\R^n} \langle z_1-x \rangle^{-N} \langle x \rangle ^{-2N} \, dx \lesssim \langle z_1 \rangle^{-N},
        \]
        where we use that $N \geq n+1$ so that $\int \langle x \rangle^{-N} \, dx < \infty$. This estimate implies 
        \[
            \sup_{w \in \R^{2n}} |\langle \pi(w) T_\sigma \pi(w)^* \varphi, \pi(z) \varphi \rangle| \leq C(n,K,N)  \langle z_1 \rangle ^{-N} \langle z_2 \rangle^{-N}
        \]
        which implies the desired result.
    \end{proof}

\begin{cor}\label{HormanderLocalization}
    Let $p,q \in [1,\infty)$, $\nu$ be an admissible weight of polynomial growth, and $m$ be a $\nu$-moderate weight. If $\sigma\in L^{\infty}(\mathbb{R}^{3n})$ with $\text{supp}\,\widehat{\sigma} \subseteq K$ for some compact $K \subseteq \mathbb{R}^{3n}$, then $T_\sigma$ is $\nu$-localized and there exists $C(K,\nu)>0$ such that 
    \[
        \|T_\sigma\|_{M_m^{p,q}(\R^n) \rightarrow M_m^{p,q}(\R^n)} \leq C(K,\nu)\|\sigma\|_{L^\infty(\mathbb{R}^{3n})}.
    \]
\end{cor}
\begin{proof}
    The $\nu$-localization is a direct consequence of Lemma \ref{PairingBound}, and the operator norm estimate follows from the quantitative dependence of the inequality in Lemma \ref{PairingBound} on $\|\sigma\|_{L^{\infty}(\mathbb{R}^{3n})}$ and the bound in the proof of Proposition \ref{localized>bounded}.
\end{proof}

\begin{prop}\label{HormanderCompactness}
    If $\sigma\in L^{\infty}(\mathbb{R}^{3n})$ with $\text{supp}\,\widehat{\sigma} \subseteq K$ for some compact $K \subseteq \mathbb{R}^{3n}$ and 
    $$
        \lim_{z \rightarrow \infty} \partial^\alpha\sigma(x-z_1,y-z_1,\xi-z_2) = 0
    $$ 
    pointwise for all multi-indices $\alpha$, then $T_\sigma$ is weakly compact.
\end{prop}
\begin{proof}
        It is easy to see that weak compactness is equivalent to
        \[
            \lim_{z \rightarrow \infty} \sup_{\substack{w\in\mathbb{R}^{2n} \\ |z-w| \leq M}} | \langle T_\sigma \pi(z) \varphi, \pi(w) \varphi \rangle | = 0,
        \]
        for each $M>0$, or equivalently
        \[
            \lim_{z \rightarrow \infty} \sup_{\substack{w\in\mathbb{R}^{2n}\\ |w| \leq M}} | \langle T_{\sigma_{z}}  \varphi, \pi(w) \varphi \rangle| = 0
        \]
        where $\sigma_{z} = \sigma(x-z_1,y-z_1,\xi-z_2)$. Put $N = \lfloor\frac{n}{2}\rfloor+1$.
        \begin{align*}
            &\langle T_{\sigma_{z}}  \varphi, \pi(w) \varphi \rangle = \iiint_{\R^{3n}} \sigma(x-z_1,y-z_1,\xi-z_2) \varphi(y) (\overline{\pi(w)\varphi})(x) e^{2 \pi i (x-y)\cdot\xi} \, dy d \xi dx\\
            &\quad= \iiint_{\R^{3n}} \frac{\varphi(y) (\overline{\pi(w) \varphi})(x)}{(1+4\pi^2 |\xi|^2)^{N}} \sigma(x-z_1,y-z_1,\xi-z_2) (I-\Delta_y)^N e^{2 \pi i  (x-y)\cdot \xi} \, dy d \xi dx.
        \end{align*}
        Upon expanding the differential operator, integrating by parts, and applying the Leibniz rule, we see that it suffices to estimate
        \[
            \iiint_{\R^n} \frac{e^{2 \pi i (x-y)\cdot \xi}}{(1+4\pi^2 |\xi|^2)^{N}} \partial_y^\alpha \sigma(x-z_1,y-z_1,\xi-z_2) \partial_y^\beta \varphi(y) (\overline{\pi(w) \varphi})(x) \, dy d \xi dx,
        \]
        where $|\alpha|,|\beta| \leq 2N$. Since the Schwartz semi-norms of $\overline{\pi(w) \varphi}$ are bounded uniformly in $|w| \leq M$, we can estimate the integrand by a constant times 
        \[
            (1+|\xi|^2)^{-N}(1+|x|^2)^{-N}(1+|y|^2)^{-N} \left| \partial_y^\alpha \sigma(x-z_1,y-z_1,\xi-z_2) \right|.
        \]
        Applying the decay assumption on $\sigma$ and dominated convergence, we conclude that
        \[
            \lim_{z \rightarrow \infty} \sup_{\substack{w\in\mathbb{R}^{2n}\\ |w| \leq M}} | \langle T_{\sigma_{z}}  \varphi, \pi(w) \varphi \rangle| = 0,
        \]
        as required.
\end{proof}

We next establish an atomic decomposition for weighted Sj\"ostrand classes, which we use to pass from bounded symbols with compact Fourier support to general symbols in $M_{1 \otimes \nu}^{\infty,1}(\R^n)$. This result is based on the corresponding unweighted result of \cite{B1997}*{Theorem 1.2}.

\begin{thm}\label{AtomicDecomp}
    There exists a compact set $K\subseteq\mathbb{R}^{n}$ such that for any continuous, submultiplicative weight $\nu$ of polynomial growth and $\sigma \in M_{1\otimes \nu}^{\infty,1}(\R^n)$, we have
    \[
        \sigma = \sum_{k \in \Z^{n}} \sigma_k,
    \]
    where $\text{supp} \, \widehat{\sigma_k} \subseteq k+K$ and $\|\sigma_k\|_{L^{\infty}(\mathbb{R}^{n})} \lesssim \nu(k)^{-1} L(k)$ for some $L$ satisfying $ \sum_k L(k) = \|\sigma\|_{M_{1 \otimes \nu}^{\infty, 1}(\R^n)}$.
    
\end{thm}
\begin{proof}
    For $t \in \R^n$, define $\widehat{\theta_t}(\cdot) := \phi(\cdot - t)$, where $ \phi \in C_0^\infty(\mathbb{R}^n)$ has the property 
    \[
        \sum_{k \in \Z^n} \phi(\cdot - k) \equiv 1.
    \]
    Then, for any $t \in Q := [0,1)^n$, we have 
    \[
        \sigma = \sum_{k \in \Z^n} \sigma * \theta_{k+t}.
    \]
    Put $\psi := \int_Q \theta_{t} \, dt$ and $\psi_k(x) := e^{2 \pi i x \cdot k} \psi(x) 
    = \int_Q e^{2 \pi i (k + t) \cdot x} \theta(x) \, dt$. Then $\text{supp} \, \widehat{\psi_k} \subseteq k+K$. Taking an average in $t$, we see
    \[
        \sigma = \sum_{k \in \Z^n} \sigma_k,
    \]
    where $\sigma_k := \sigma * \psi_k$. It remains to prove the $L^\infty(\mathbb{R}^n)$ estimate on $\sigma * \psi_k$: 
    \begin{align*}
        |\sigma * \psi_k (x)| &= \bigg| \int_Q \int_{\R^n} \sigma(x-y) e^{2 \pi i y \cdot (k+t)} \theta(y) \, dy dt \bigg|\\
        &\leq \int_Q \left| \int_{\R^n} \sigma(y) e^{-2 \pi i y \cdot (k+t)} \theta(x-y) \, dy \right| dt\\
        &= \int_Q |\langle \sigma, M_{k+t} T_x \tilde{\overline{\theta}} \rangle| \, dt.
    \end{align*}
    Since $\nu$ is continuous and submultiplicative, we have $\nu(k) \lesssim \nu(k+t)$ for $t \in Q$, and so
    \[
        \sup_{x \in \mathbb{R}^n} \int_Q |\langle \sigma, M_{k+t} T_x \tilde{\overline{\theta}} \rangle| \, dt \lesssim \nu(k)^{-1} \int_Q \sup_{x\in\mathbb{R}^n} |\langle \sigma, M_{k+t} T_x \tilde{\overline{\theta}} \rangle| \, \nu(k+t) dt.
    \]
    Since $\tilde{\overline{\theta}} \in \mathcal{S}(\R^n)$ and $\sigma \in M_{1 \otimes \nu}^{\infty,1}(\mathbb{R}^n)$,
    \[
        \sum_{k \in \mathbb{Z}^n} \int_Q \sup_{x\in\mathbb{R}^n} |\langle \sigma, M_{k+t} T_x \tilde{\overline{\theta}} \rangle| \, \nu(k+t) dt = \int_{\R^{2n}} \sup_{x\in\mathbb{R}^n} |\langle \sigma, M_{t} T_x \tilde{\overline{\theta}} \rangle| \, \nu(t) dt \approx \|\sigma\|_{M_{1 \otimes \nu}^{\infty,1}(\R^n)},
    \]
      which gives the desired conclusion.
\end{proof}

The following bound allows us to conclude that three-parameter pseudodifferential operators with atomic symbols are $\nu$-localized. 
\begin{lemma}\label{AtomBound}
    If $\sigma \in L^\infty(\R^{3n})$ with $\text{supp} \, \widehat{\sigma} \subseteq k + K$ for some compact $K\subseteq\mathbb{R}^{3n}$ and $k = (k_1,k_2,k_3) \in \Z^{3n}$, then 
    \[
        \sup_{w \in \R^{2n}} |\langle \pi(w)T_{\sigma} \pi(w)^* \varphi, \pi(z) \varphi \rangle| \lesssim \|\sigma\|_{L^{\infty}(\mathbb{R}^{3n})}\langle z + (k_3, -k_1-k_2) \rangle^{-N}
    \]
    for all $z \in \mathbb{R}^{2n}$ and all $N \in\mathbb{N}$.
\end{lemma}
\begin{proof}
    First, note that conjugation by $\pi(w)$ results in translation of the symbol, which doesn't affect the hypotheses, so we can ignore the $\pi(w)$ terms. Let $\tilde{\sigma}(x) := e^{-2 \pi i x \cdot k} \sigma(x)$. Then $\text{supp} \,\widehat{\tilde{\sigma}} \subseteq K$ and we compute
    \begin{align*}
        |\langle T_\sigma \varphi, \pi(z) \varphi \rangle| &= \iiint_{\R^{3n}} \sigma(x,y,\xi)e^{2\pi i(x-y)\cdot \xi} \varphi(y) \overline{ \pi(z) \varphi(x)}\,dyd\xi dx\\
    &= \iiint_{\R^{3n}} \tilde{\sigma}(x,y,\xi) e^{2 \pi i k_1 \cdot x} e^{2 \pi i k_2 \cdot y} e^{2 \pi i \xi \cdot k_3} e^{2\pi i(x-y)\cdot \xi} \varphi(y) \overline{\pi(z) \varphi(x)}\,dyd\xi dx\\
    &= \iiint_{\R^{3n}} \tilde{\sigma}(x,y,\xi) e^{2\pi i(x + k_3 -y)\cdot \xi}(M_{k_2} \varphi)(y) \overline{(M_{-k_1} \pi(z) \varphi)(x)}\,dyd\xi dx\\
    &= \iiint_{\R^{3n}} \tilde{\sigma}(x-k_3,y,\xi) e^{2\pi i(x-y)\cdot \xi}(M_{k_2} \varphi)(y) \overline{(T_{k_3} M_{-k_1} \pi(z) \varphi)(x)}\,dyd\xi dx.
    \end{align*}
    Note that $\tilde{\sigma}(x-k_3, y , \xi)$ satisfies the hypotheses Lemma \ref{PairingBound}, and so we conclude
    \[
        |\langle T_{\sigma} \varphi, \pi(z) \varphi \rangle| \lesssim \|\sigma\|_{L^{\infty}(\mathbb{R}^{3n})}\langle z + (k_3, -k_1-k_2) \rangle^{-N}.
    \]
\end{proof}

\begin{proof}[Proof of Theorem \ref{GeneralPDOCompactness}] 
    We first show that $T_{\sigma}$ is $\nu$-localized. By Theorem \ref{AtomicDecomp}, we have the representation
    \[
        T_\sigma = \sum_{k \in \Z^{3n}} T_{\sigma_k}
    \]
    where $\text{supp} \,\widehat{\sigma_k} \subset k + K$ for some compact $K\subseteq\mathbb{R}^{3n}$ and $\|\sigma_k\|_{L^\infty(\mathbb{R}^{3n})} \leq \nu(-k_3,k_1+k_2)^{-1} L(k)$ with $\sum_k L(k) < \infty$. By Lemma \ref{AtomBound}, we have that
    \begin{align*}
        \sup_{w \in \R^{2n}} |\langle \pi(w)T_{\sigma_k} \pi(w)^* \varphi, \pi(z) \varphi \rangle| \nu(z) &\lesssim \nu(-k_3,k_1+k_2)^{-1} L(k) \nu(z) \langle z + (k_3, -k_1-k_2) \rangle^{-N}\\
        &\leq L(k) \nu(z + (k_3, -k_1 - k_2)) \langle z + (k_3, -k_1 - k_2) \rangle^{-N}
    \end{align*}
    which has integral bounded by $L(k)$ for sufficiently large $N$. Since $L$ is summable, it follows that
    \[
        \int_{\R^{2n}} \sup_{w \in \R^{2n}} |\langle \pi(w)T_{\sigma} \pi(w)^* \varphi, \pi(z) \varphi \rangle| \nu(z) \, dz < \infty,
    \]
    which implies that $T_\sigma$ is $\nu$-localized.

    We next show that $T_{\sigma}$ is weakly compact under the additional decay assumption -- the result then follows from Theorem \ref{AbstractCompactnessTheorem}. Write $\sigma = \sum_k \sigma_k$ as in Theorem \ref{AtomicDecomp}. From the argument above and Proposition \ref{localized>bounded}, we see that $\sum_k \|T_{\sigma_k}\|_{L^2 \rightarrow L^2} < \infty$, and so, by the dominated convergence theorem, it suffices to prove that each $T_{\sigma_k}$ is weakly compact. But for any multi-index $\alpha$, we have 
    $$ 
        \partial^\alpha \sigma_k(x-z_1,y-z_1,\xi-z_2) = \sigma * \partial^\alpha \psi_k(x-z_1,y-z_1,\xi-z_2),
    $$ which goes to zero pointwise by our decay assumption on $\sigma$ since $\psi_k \in \mathcal{S}(\mathbb{R}^n)$. Therefore, each $T_{\sigma_k}$ is weakly compact by applying Proposition \ref{HormanderCompactness} with $k+K$ in place of $K$.
\end{proof}


\section{Acknowledgements}
The authors thank Robert Fulsche, Karlheinz Gr\"ochenig, Mishko Mitkovski, and Kasso Okoudjou for inspiring discussions and feedback.


\end{document}